\documentclass{article}
\usepackage{amsmath}
\usepackage{amssymb}
\usepackage{amsthm}
\usepackage{xifthen}
\usepackage{txfonts}
\usepackage{turnstile} 
\usepackage{bbm} 
\usepackage{enumerate} 
\usepackage{mathtools}
\usepackage{float}
\usepackage{tikz-cd}
\usepackage{tikz}
\usepackage{circuitikz}
\usetikzlibrary{positioning, arrows.meta}
\DeclareFontFamily{U}{mathx}{}
\DeclareFontShape{U}{mathx}{m}{n}{<-> mathx10}{}
\DeclareSymbolFont{mathx}{U}{mathx}{m}{n}
\DeclareMathAccent{\widehat}{0}{mathx}{"70}
\DeclareMathAccent{\widecheck}{0}{mathx}{"71}

\newtheorem{thm}{Theorem}[section]
\newtheorem{lemma}[thm]{Lemma}
\newtheorem{prop}[thm]{Proposition}
\newtheorem{cor}[thm]{Corollary}
\newtheorem{quest}{Question}

\theoremstyle{definition}
\newtheorem{dfn}[thm]{Definition}

\theoremstyle{remark}

\newtheorem*{rmk}{Remark}

\newtheorem*{claim}{Claim}
\newenvironment{claimproof}[1]{\par\noindent\textit{Proof:}\space#1}{\leavevmode\unskip\penalty9999 \hbox{}\nobreak\hfill\quad\hbox{$\blacksquare$}}

\newcommand{\ZF}{\text{ZF}}

\newcommand{\DC}{\text{DC}}
\newcommand{\HOD}{\text{HOD}}
\newcommand{\OD}{\text{OD}}
\newcommand{\reals}{\mathbb{R}}
\newcommand{\RR}{\mathbb{R}}

\newcommand{\zahlen}{\mathbb{Z}}
\newcommand{\rationals}{\mathbb{Q}}

\DeclareMathOperator{\dom}{\text{dom}}
\newcommand{\nat}{\mathbb{N}}
\newcommand{\Coll}[2]{\text{Coll}({#1}, {#2})}

\newcommand{\balanceEq}{\equiv_b}
\newcommand{\forc}[2]{\dststile{#2}{#1}} 
\DeclareMathOperator{\OR}{OR}
\newcommand{\tup}[1]{\langle#1\rangle}
\newcommand{\tupp}[1]{\left\langle#1\right\rangle}
\newcommand{\set}[1]{\{#1\}}
\DeclareMathOperator{\lh}{lh}

\newcommand{\ptwiseImg}{\mathbin{''}}
\DeclareMathOperator{\Uf}{\mathsf{Uf}}

\makeatletter
\newcommand{\sqsubsetsim}{\mathrel{\mathpalette\su@sim\sqsubset}}
\newcommand{\sqsupsetsim}{\mathrel{\mathpalette\su@sim\sqsupset}}
\newcommand{\nsqsubsetsim}{\mathrel{\mathpalette\su@sim\nsqsubset}}
\newcommand{\nsqsupsetsim}{\mathrel{\mathpalette\su@sim\nsqsupset}}
\newcommand{\su@sim}[2]{%
  \begingroup
  \sbox\tw@{$\m@th#1\mkern-1mu\sqsubseteq\mkern-1mu$}%
  \raisebox{\dimeval{\ht\tw@-\height}}{%
    \vbox{\m@th\offinterlineskip
      \sbox\z@{$#1\mkern1mu$}
      \ialign{%
        \hfil##\hfil\cr
        $#1#2$\cr
        \noalign{\vskip \wd\z@}
        \ifx#2\supset
          \reflectbox{\resizebox{\wd\tw@}{!}{$#1\sim$}}%
        \else
          \resizebox{\wd\tw@}{!}{$#1\sim$}%
        \fi\cr
        \noalign{\vskip -2\wd\z@}
      }%
    }%
  }%
  \endgroup
}
\makeatother
\newcommand{\pbelow}{\sqsubsetsim} 

\newcommand{\npbelow}{\nsqsubsetsim}

\newcommand{\psbelow}{\sqsubset} 

\newcommand{\pless}{\lesssim} 
\newcommand{\psless}{<} 

\newcommand{\PrelinPoset}[1][] 
{
  \ifthenelse{\isempty{#1}}
    {\mathbf{P}}
    {\mathbf{P}({#1})}
}

\newcommand{\Pprelin}{\mathbf{P}_{\pbelow}} 

\usepackage{xcolor}

\usepackage{hyperref}
\definecolor{mypurple}{HTML}{7f00d4}
\definecolor{mygreen}{HTML}{71a300}
\hypersetup{
	colorlinks=true,
	linkcolor=mypurple,
	filecolor=mypurple,
	citecolor=mygreen,
	urlcolor=black,
}

\title{Funicular Preorders Can Be Prelinearized Without Nonprincipal Ultrafilters over $\mathbb N$}
\author{Azul Fatalini \and Luke Serafin}
\date{}

\begin{document}

\maketitle

\begin{abstract}
It is a consequence of the axiom of choice that every preorder can be extended to a total preorder while respecting the strict preorder relation. We call such an extension a prelinearization of the preorder and study the extent to which the axiom of choice is needed to construct prelinearizations. 
We isolate the class of funicular preorders, and show that these have prelinearizations in models of $\mathsf{ZF+DC}$ containing no nonprincipal ultrafilters over $\omega$. Funicular preorders include coordinatewise domination on $\RR^\omega$, Turing reducibility, and various preorders arising in social choice theory. The relevant models are constructed first using tools from the geometric set theory of Larson and Zapletal, which requires an inaccessible cardinal, and then the inaccessible cardinal is eliminated using methods of amalgamation for Cohen reals.

\end{abstract}

{\def\thefootnote{}\footnotetext{The first author was supported by UKRI Future Leaders Fellowship [MR/Y034058/1]. The second author was supported by NSF grants DMS-2153975 and DMS-2451350.}}

\section{Introduction}

An important aspect of choiceless set theory is gauging the logical strength of consequences of the axiom of choice under $\mathsf{ZF}$. 
If $A$ and $B$ are two statements that are provable in $\mathsf{ZFC}$, without the axiom of choice they could be equivalent, incomparable, or one may imply the other and not the other way around.
Implication and equivalence can generally be proven outright, but mathematically establishing incomparability under implication or the strictness of an implication usually requires constructing models of ZF in which one statement holds but not the other.
Until recently most work on models of the negation of the axiom of choice was focused on establishing that particular statements are consistent with ZF (perhaps augmented by large cardinals); examples of such statements are the axiom of determinacy~\cite{martin-steel-PD-proof} and the statement that $\aleph_1$ is a singular cardinal (the Feferman-Lévy model, for which see~\cite[Sec. 10.1]{jech-ac}).

Strategies to construct and verify witness models for the fact that $A$ does not imply $B$ under ZF only are an active area of investigation. 
The book \emph{Geometric Set Theory}~\cite{larson-zapletal-GST} is partially dedicated to the construction of methods and criteria for this purpose, and it has been applied successfully in a variety of different contexts; see, for example,~\cite{panagiotopoulos-shani-strong-ergodicity, serafin-SEA, zapletal-sequential-dedekind-finite, zapletal-coloring-distance-graphs, zapletal-coloring-triangles-rectangles, zapletal-krull, zapletal-coloring-noetherian, zapletal-triangles-vitali, zhou-coloring-isosceles-triangles-choiceless}. 
Also, there are results achieved using other techniques, for example~\cite{Beriashvili2022, Brendle2018, fatalini-unit-circle-partition, Schilhan2026}. 

In this paper we focus on statements $A$ of the form \emph{``there is a prelinearization of $\pbelow$''}, where $\pbelow$ is a fixed, definable preorder (see Definition~\ref{dfn:prelinearization}) and $B$ is the statement that there exists a nonprincipal ultrafilter over $\omega$ (denoted by $\mathsf{Uf}(\omega)$).

The economics subfield of social choice theory furnishes motivation for studying the conditions under which particular preorders can be prelinearized, as pointed out in the long-unpublished paper ``Flutters and Chameleons''~\cite{flutters-chameleons}.
This paper lays out the mathematical theory needed to solve a variety of problems arising in social choice theory, and the companion paper \cite{economics-paper} shows how the mathematical theory can be used to show that various constructions in social choice theory traditionally done with nonprincipal ultrafilters over $\omega$ or stronger principles consistently do not need the ultrafilters.

Coming back to the set-theoretic perspective, the second author has proved that there is a model of $\mathsf{ZF+DC}$ with a prelinearization of a preorder $\pbelow_{\text{SEA}}$ arising in the context of equity in social choice theory, and no nonprincipal ultrafilter on $\omega$~\cite{serafin-SEA}.
This result uses geometric set theory methods, more precisely the fact that the natural poset to add a prelinearization by countable approximations is \emph{placid}, one of the notions of geometric set theory.
However, not all the forcing notions associated to prelinearization of preorders are placid, and in particular we observe (Proposition~\ref{prop: coordinatewise is not tranquil}) that the coordinatewise domination order on $\RR^\omega$ lacks a technical property needed to verify that its natural forcing notion is placid, so other methods were needed.
The coordinatewise domination order is of special importance in social choice theory because it encodes the Pareto efficiency condition.

Motivated by this prior result, we want to understand which prelinearizations do not require $\mathsf{Uf}(\omega)$ (in $\mathsf{ZF} + \mathsf{DC}$).
We isolated a concept that helps us obtain a partial answer to this question: if a Borel preorder $\pbelow$ is \emph{funicular} (Definition~\ref{def:funicular}) and there is an inaccessible cardinal, then there is a model of $\mathsf{ZF+DC+\neg Uf(\omega)}$ with a prelinearization of $\pbelow$.
The notion of $(3,2)$-balanced forcing, as introduced in~\cite{larson-zapletal-GST}, is used to ensure that the model constructed satisfies $\neg \mathsf{Uf}(\omega)$, by proving that funicular preorders are naturally associated to $(3,2)$-balanced forcings (Lemma \ref{lemma: funicular implies 32balanced}). 

However, it is always desirable to understand whether we can take the inaccessible away~\cite{shelah-take-solovays-inaccessible-away}.
The geometric set theory methods are usually used in the context of forcing over the Solovay model, which is of necessity constructed using an inaccesible cardinal. 
To take the inaccesible away we need then to solve two problems:
\begin{enumerate}
\item Formulate a similar criterion to $(3,2)$-balanced in a context that does not use an inaccessible cardinal, and
\item Prove that our family of funicular Borel preorders have associated forcing notions satisfying the criterion. 
\end{enumerate}
These are Theorem~\ref{thm:local-prelinearization} and Lemma~\ref{lemma: local forcing forces the object}, respectively.

As described before, the paper covers both an argument directly using methods from geometric set theory and a substantial modification of this argument, which removes the need for using an inaccessible cardinal.
The following diagram shows the relation between the subsequent sections or subsections. 
The reader interested in following only one of the paths should follow the corresponding track drawn in the diagram.

\begin{figure}[!h]
\centering
\resizebox{0.7\textwidth}{!}{%
\begin{circuitikz}
\tikzstyle{every node}=[font=\LARGE]
\node [font=\LARGE] at (9,12.25) {};

\draw  (9.25,13.75) rectangle (9.25,13.75);
\node [font=\LARGE] at (9.5,12.25) {};
\draw  (8.25,14.25) rectangle (8.25,14.25);
\draw  (13.25,17.25) rectangle (13.25,17.25);
\draw  (19,12.75) rectangle (19,12.75);
\draw  (16.25,12.5) rectangle (16.25,12.5);
\draw  (16.25,12.5) rectangle (16.25,12.5);
\draw  (16.25,12.5) rectangle (16.25,12.5);
\draw  (16.25,12.5) rectangle (16.25,12.5);
\node [font=\LARGE] at (13.25,15.25) {2.1 Order relations \& 2.2 Descriptive set theory};
\draw  (6.25,16) rectangle (20.5,14.5);
\draw (8.75,14.5) to[short] (8.75,12.25);
\draw (18.25,14.5) to[short] (18.25,12.25);
\node [font=\LARGE] at (18.25,11.5) {2.3 Cohen forcing};
\node [font=\LARGE] at (9.25,11.5) {2.4 Geometric set theory};
\draw (8.75,11) to[short] (8.75,8.5);
\draw (18.25,11) to[short] (18.25,8.5);
\node [font=\LARGE] at (13.25,7.75) {3. Funicular orders };
\draw  (6.25,8.5) rectangle (20.5,7);
\draw  (21.5,7.5) rectangle (21.5,7.5);
\draw (8.75,7) to[short] (8.75,4.75);
\draw (18.25,7) to[short] (18.25,4.75);
\node [font=\LARGE] at (18.25,4.25) {5. Eliminating the inaccessible};
\node [font=\LARGE] at (9,4.25) {4. Forcing over the Solovay model};
\end{circuitikz}
}%

\label{fig:my_label}
\end{figure}

\section{Preliminaries}

We identify functions and relations with their graphs, and write $f \restriction A$ for the restriction of a function to a set $A$.
So if $f \mathrel{:} X \rightarrow Y$ and $A \subseteq X$, then $f \restriction A = f \cap (A \times Y)$.
The pointwise image of a set $A$ under a function $f$ is denoted $f \ptwiseImg A$.
For $X$ and $Y$ sets, $X^Y$ denotes the space of all functions with domain $Y$ and codomain $X$.
The notation $X^{<\omega}$ indicates the space of all finite sequences of elements of $X$; note that (up to obvious bijections, depending on the precise definitions used) $X^{<\omega} = \bigcup_{n \in \omega} X^n$.
For $s \in X^{<\omega}$ the notation $|s|$ denotes the length of the sequence $s$, which is indeed the cardinality of $s$ as a set of ordered pairs.

\subsection{Order relations}

A \emph{binary relation} on a set $X$ is a subset $R$ of $X \times X$ (the Cartesian product of $X$ with itself).
For $x, y \in X$ we write $x \mathrel{R} y$ when $\langle x, y \rangle \in R$.
Let the variables $x, y$, and $z$ denote arbitrary elements of $X$ for the moment.
A binary relation $R$ is
\begin{itemize}
\item \emph{transitive} when $x \mathrel{R} y \wedge y \mathrel{R} z \Rightarrow x \mathrel{R} z$,
\item \emph{reflexive} when $x \mathrel{R} x$,
\item \emph{irreflexive} when $\neg x \mathrel{R} x$,
\item \emph{total} when $x \mathrel{R} y$ or $y \mathrel{R} x$,
\item \emph{antisymmetric} when $x \mathrel{R} y \wedge y \mathrel{R} x \Rightarrow x = y$,
\item \emph{symmetric} when $x \mathrel{R} y \Rightarrow y \mathrel{R} x$, and
\item \emph{asymmetric} when $x \mathrel{R} y \Rightarrow \neg y \mathrel{R} x$.
\end{itemize}
For $R$ a binary relation on a set $X$, the pair $(X, R)$ is called a \emph{relational structure}.
A relational structure $(X, R)$ is
\begin{itemize}
\item a \emph{preorder} when $\mathrel{R}$ is reflexive and transitive,
\item a \emph{partial order} when $\mathrel{R}$ is an antisymmetric preorder,
\item a \emph{linear order} when $\mathrel{R}$ is a total partial order,
\item a \emph{prelinear order} when $\mathrel{R}$ is a total preorder, and
\item an \emph{equivalence relation} when $\mathrel{R}$ is a symmetric preorder.
\end{itemize}

It is easy to see that a preorder $\pbelow$ induces an equivalence relation $\sim_{\pbelow}$ defined by $x \sim_{\pbelow} y$ if and only if $x \pbelow y$ and $y \pbelow x$.
When the preorder $\pbelow$ is clear from the context, we write simply $x \sim y$ in place of $x \sim_{\pbelow} y$.
Also, given a preorder $\pbelow$, the corresponding \emph{strict} relation $\psbelow$ is defined by $x \psbelow y$ if and only if $x \pbelow y$ and $y \npbelow x$.
This is an irreflexive, asymmetric, and transitive relation.

A binary relation $S$ \emph{extends} a binary relation $R$ precisely when $R \subseteq S$.
We now come to the main object of study of this paper.
\begin{dfn} \label{dfn:prelinearization}
Let $\pbelow$ be a preorder on a set $X$.
An extension $\pless$ of $\pbelow$ is a \emph{prelinearization} of $\pbelow$ precisely when $\pless$ is a prelinear order on $X$, ${\pbelow} \subseteq {\pless}$, and ${\psbelow} \subseteq {\psless}$.

A prelinearization $\pless$ of $\pbelow$ is a \emph{classwise prelinearization} when furthermore ${\sim_{\pless}} = {\sim_{\pbelow}}$.
\end{dfn}
Note that a prelinearization of a preorder $\pbelow$ is allowed to identify incomparable $\sim_{\pbelow}$-equivalence classes, while this is not allowed for a classwise prelinearization.
The axiom of choice may be used to show that classwise prelinearizations always exist, while under the axiom of dependent choice there are partial orders which can be prelinearized but not linearized
(and hence cannot be prelinearized classwise, because the equivalence relation induced by a partial order is the identity relation).

\begin{thm}[$\mathsf{ZFC}$; Szpilrajn~\cite{szpilrajn-extension} and Hansson~\cite{hansson-choice-preference}] \label{lemma: prereq preorder can be extended to prelin}
    Let $\pbelow$ and $\pless$ two preorders on a set $X$. If ${\pbelow} \subseteq {\pless}$ and ${\psbelow} \subseteq {\psless}$, then $\pless$ can be extended to a classwise prelinearization of $\pbelow$.
\end{thm}

To see that the Szpilrajn-Hansson theorem just stated need not hold under dependent choice, note that Solovay's model $W$ satisfies DC~\cite{solovay}, but there is no prelinearization of the relation of subset on sets of reals. This is because the existence of such would imply the existence of a linear order of $E_0$ classes\footnote{Here $E_0$ is the standard relation of eventual equality between elements of $\omega^\omega$; see e.g.~\cite{gao-invariant-descriptive}.}, which cannot exist in $W$ since the inverse image of such a linear order under the quotient projection from $\omega^\omega$ to $\omega^\omega / E_0$ is a set of reals without the property of Baire by an application of the Kuratowski-Ulam theorem, while every set of reals in $W$ has the property of Baire.

\subsection{General and descriptive set theory}

A \emph{Polish space} is a separable topological space for which there exists a complete metric generating the topology.

Note that countable disjoint unions and products of Polish spaces are Polish (see e.g.~\cite{kechris-classical}). 
A subset of a Polish space is $G_\delta$ when it can be expressed as a countable intersection of open sets, and
the subset is \emph{Borel} when it is an element of the Borel $\sigma$-algebra generated by the open sets.
A function between Polish spaces is Borel when the inverse image of an open subset of the codomain is Borel as a subset of the domain.
A subspace of a Polish space is Polish precisely when it is $G_\delta$~\cite{kechris-classical}. 
A subset $A$ of a Polish space $X$ is \emph{analytic} or $\mathbf{\Sigma}^1_1$ when there exists a Polish space $Y$ and a continuous function $f \mathrel{:} Y \rightarrow X$ such that $A$ is the image of $Y$ under $f$.

We use $\ZF$ as our foundational set theory, usually augmented by the axiom of dependent choice.
In particular, we do not assume the axiom of choice unless specifically noted.
Our arguments take place in the context of a background model which is assumed to satisfy the axiom of choice, so at times we do make use of the full axiom of choice.
For details on the axioms see~\cite[Ch. 0]{jech}. For completeness, we state the axiom of dependent choice here.

\begin{dfn}
The \emph{axiom of dependent choice} denoted by $\mathsf{DC}$ states that for every binary relation $R$ with the property that for every $x \in \dom(R)$ there is $y$ such that $x \mathrel{R} y$,
there is a function $f : \omega \rightarrow \dom(R)$ with the property that for every $n \in \omega$, $f(n) \mathrel{R} f(n+1)$.
\end{dfn}

See a standard text on set theory such as~\cite{kunen} for basic facts about forcing over models of set theory.
We use the Boolean convention for the forcing order, so $p \le q$ means that $p$ is a stronger condition than $q$.
For $p$ and $q$ forcing conditions, $p \perp q$ means that $p$ and $q$ are incompatible (that is, there is no condition $r$ with $r \le p, q$).
If $P$ is a notion of forcing, $\dot a$ is a $P$-name, and $g$ is a $P$-generic filter over a model of set theory $V$, we denote by $\dot a_g$ the interpretation of $\dot a$ by $g$.
For $g$ a generic filter in a product $\prod_{\alpha < \kappa} P_\alpha$ and $\alpha < \kappa$, we denote by $g \restriction \alpha$
the set $\{ p \restriction \alpha : p \in g \}$.

Regarding inner models, we assume that the reader is familiar with G\"odel's constructible universe $L$, its variants of the form $L(A)$,
the class of ordinal-definable sets $\OD$ and its relativization $\OD_A$ (note these are not in general models of set theory), and the class $\HOD$ of ordinal-definable sets and its relativization $\HOD_A$

\subsection{Cohen forcing}
Familiarity with Cohen forcing is assumed, but in order to elucidate some subtleties and establish notation that will be needed in Section~\ref{section: no inaccessible}, we record some important definitions and results here.
In this subsection the full axiom of choice is assumed.

\begin{dfn}\label{def: prereq cohen forcing and C(X)}
    The forcing $\mathbf{C}$, called \emph{Cohen forcing}, is defined as follows.
	\[	\mathbf{C} = \{p\colon \omega \to 2 \mid  p \text{ is a partial function with } \dom(p)<\omega \},\]
	ordered by reverse inclusion, with $\mathbbm{1}_{\mathbf{C}}=\emptyset$.
    Moreover, when $X$ is a set of ordinals we define
	\[\mathbf{C}(X)= \left\{p \in \Pi_{\alpha\in X} \mathbf{C} : \left| \{ \alpha\in X : p(\alpha)\neq \emptyset \} \right| <\omega  \right\}, \]
	ordered coordinatewise.
    This is the finite-support product of $|X|$-many copies of Cohen forcing.
\end{dfn}

For $V$ a model of ZFC and $g \subseteq \mathbf{C}$ a generic filter over $V$,
$\cup g \in {^\omega}\omega \cap V[g]$ is a real number called the (canonical) \emph{Cohen real} added by $\mathbf C$.
The adjective ``canonical'' is needed because Cohen forcing in fact adds uncountably many reals which are unions of generic filters for forcings equivalent to Cohen forcing.

If $g$ is $\mathbf{C}(X)$-generic over $V$, we have for each $\alpha \in X$ that $\bigcup_{p \in g} p(\alpha)$ is a real for each $\alpha \in X$; in fact it is a Cohen real.
Abusing notation, we often treat a Cohen-generic filter $g \subseteq \mathbf{C}$ as though it were a real number, rather than explicitly writing $\cup g$.

In the special case of $\mathbf{C}(\omega_1)$, every proper initial segment of the product
is equivalent to one single instance of Cohen forcing $\mathbf{C}$ (because it is a countable product),
and in fact every real in the $\mathbf{C}(\omega_1)$-extension is added by a
proper initial segment, a fact which we shall use frequently.

\begin{lemma} \label{lemma: prereq any real is in an inital segment of C(omega1)}
    For any $\mathbf{C}(\omega_1)$-generic filter over $V$ and
    $r \in {^\omega}\omega \cap V[g]$ there is $\alpha < \omega_1$ such that
    $r \in V[g \restriction \alpha]$.
\end{lemma}

	Notice that $g\restriction \alpha$ is $\mathbf{C}(\alpha)$-generic over $V$; so the notation $V[g\restriction \alpha]$ makes sense.
    The model $V[g \restriction \alpha]$ is in fact a Cohen-generic extension of $V$
    because $g \restriction \alpha$ is $\mathbf{C}(\alpha)$-generic over $V$ and $\alpha$ is countable.
    Note also that there are isomorphisms of partial orders between $\mathbf{C}(\omega_1)$
    and the posets $\mathbf{C}(\alpha) \times \mathbf{C}(\omega_1 \setminus \alpha)$
    and $\mathbf{C}(X)$ for any ordinal $\alpha < \omega_1$
    and set $X$ of cardinality $\aleph_1$.

We also record a useful characterization of Cohen forcing in terms of forcing equivalence,
where two partial orders are forcing-equivalent precisely when their Boolean completions
are isomorphic.
\begin{thm} \label{thm:Cohen-properties}
\phantom{Introduction to the theorem, for spacing.}
\begin{enumerate}[a.]
\item If $P$ is a countable, atomless notion of forcing then $P$ is forcing-equivalent to
$\mathbf{C}$~\cite[Theorem 1 Section 4.5]{Grigorieff1975}.
\item A countable product of Cohen forcing notions is forcing-equivalent to $\mathbf{C}$.
\item \label{th: prereq any real in cohen ext is cohen} For $g$ a $\mathbf{C}$-generic filter 
over $V$ and $r \in {^\omega}\omega \cap V[g]$,
each of the extensions from $V$ to $V[r]$ and from $V[r]$ to $V[g]$ is either trivial
or a Cohen extension.
\item \label{th: prereq any real in Q extension is Q ground} For $g$ a
$\mathbf{C}(\omega_1)$-generic filter over $V$ and
$r \in {^\omega}\omega \cap V[g]$, $V[g]$ is a $\mathbf{C}(\omega_1)$-generic extension of
$V[r]$ and $V[r]$ is either a trivial or a Cohen extension of $V$.
\end{enumerate}
\end{thm}

\subsection{Geometric set theory}

The \emph{geometric set theory} of Larson and Zapletal~\cite{larson-zapletal-GST} provides a collection of techniques for forcing over the Solovay model $W$, and will be applied here in Section~\ref{sec:solovay-model-prelinearization}. 
The importance of the Solovay model is that it is a model of $\ZF + \DC$ in which all sets of reals have nearly all regularity properties of mathematical significance.
In particular, in $W$ all sets of reals are Lebesgue-measurable and have the Baire property, the perfect set property, and the Ramsey property~\cite{solovay,happy-families}. 
Note that constructing the Solovay model requires the existence of one inaccessible cardinal, and so consistency proofs using the Solovay model are not consistency proofs relative to ZFC.
To be definite we fix an inaccessible cardinal $\lambda$ and define $W$ to be $\HOD^{V[g]}_{\reals \cup V}$, where $g$ is $\Coll{\omega}{\lambda}$-generic over $V$ and $\reals$ is evaluated in $V[g]$.
A number of alternative definitions are employed in the literature, but they all yield elementarily-equivalent models~\cite{di-prisco-todorcevic-lru}.

\subsubsection{Suslin forcing and virtual conditions}

To make use of the machinery of geometric set theory, and in particular balanced forcing and its variants, it will be necessary to work
with forcing conditions in generic extensions.
In order for this to make sense, the forcing poset needs to be sufficiently definable; the notion of a \emph{Suslin forcing} is adequate for geometric set theory~\cite{larson-zapletal-GST} and our uses of it.

\begin{dfn}
A poset $\langle P, \le \rangle$ is \emph{Suslin} if and only if there is a Polish space $X$ over which $P$, $\le$, and $\perp$ are analytic.
\end{dfn}

The utility of the analyticity assumption is that Mostowski absoluteness~\cite[Theorem 25.20]{jech} applies, and so analytic sets have a meaningful interpretation in arbitrary models of $\ZF + \DC$ where they exist.
 More details about the absoluteness properties of Suslin posets and numerous examples of these posets can be found in~\cite{larson-zapletal-GST} and its references. 
 
Virtual conditions are, intuitively, objects which exist in $V$ and describe conditions of a Suslin forcing that are guaranteed to be
consistent across forcing extensions.
To formalize this, we start by defining objects in $V$, called $P$-pairs, which determine $P$-conditions in generic extensions of $V$.
Actually, we shall let $P$-pairs determine analytic sets of conditions in $P$, and for this we define an ordering on analytic subsets
of $P$ which is best thought of as ordering them by their suprema in a definable completion\footnote{That is, a minimal extension of the order $P$ which contains suprema for all analytic subsets of $P$.}.

\begin{dfn}[{\cite[Definition 5.1.4]{larson-zapletal-GST}}]
For $A, B$ analytic subsets of a Suslin forcing $P$, the supremum of $A$ is below the supremum of $B$, denoted $\sum A \le \sum B$,
if and only if every condition below an element of $A$ can be strengthened to a condition below an element of $B$. In case $\sum A \le \sum B$
and $\sum B \le \sum A$, we write $\sum A = \sum B$.
\end{dfn}

\begin{dfn}[{\cite[Def. 5.1.6]{larson-zapletal-GST}}]
A \emph{$P$-pair} for a Suslin forcing $P$ is a pair $\langle Q, \dot A \rangle$ where $Q \in V$ is a forcing poset and
$Q \Vdash \text{``$\dot A$ is an analytic subset of $P$''}$.
\end{dfn}

The analytic set named in a $P$-pair is not guaranteed to have stable characteristics across generic extensions, an issue which the
notion of a $P$-pin seeks to resolve.

\begin{dfn}[{\cite[Def. 5.1.6]{larson-zapletal-GST}}]
A $P$-pair $\langle Q, \dot A \rangle$ for a Suslin forcing $P$ is a \emph{$P$-pin} if and only if
$Q \times Q \Vdash \sum \dot A_\vartriangleleft = \sum \dot A_\vartriangleright$, where
\[ \dot A_\vartriangleleft = \{ \langle \sigma_\vartriangleleft, \langle p, q \rangle \rangle : \langle \sigma, p \rangle \in \tau, q \in P \} \]
is the lift of the name $\tau$ of the projection of a $Q \times Q$-generic filter to its left factor, and similarly for 
$\dot A_\vartriangleright$.
As in~\cite{larson-zapletal-GST}, one may find it useful to think of these as the left and right copies of the name $\dot A$.
\end{dfn}

\begin{dfn}
For $P$-pins $\langle P, \dot A \rangle$, $\langle Q, \dot B \rangle$, define the relation of \emph{virtual equivalence} by
$\langle P, \dot A \rangle \equiv \langle Q, \dot B \rangle$ if and only if $P \times Q \Vdash \sum \dot A = \sum \dot B$.
Virtual conditions are equivalence classes of this relation.
\end{dfn}

That $\equiv$ is indeed an equivalence relation is established in \cite[Proposition 5.1.8]{larson-zapletal-GST}.
The intuition behind virtual conditions is that they describe (suprema of analytic sets of) conditions in a way that is independent
of the particular generic extension of $V$ under consideration.
Note that for any poset $P$ and analytic subset $A \subseteq P$, the pair $\langle 1, \check A \rangle$ (where $1$ is the trivial one-element poset) determines a virtual
condition, so in particular $P$ embeds naturally into its set of virtual conditions (observing that distinct
singletons $\{p\}$, $\{q\}$ for $p, q \in P$ determine distinct virtual conditions). 

\subsubsection{Placid, balanced, and $(3,2)$-balanced forcing}

To define and work with balanced forcing and related notions we use an equivalence relation on $P$-pairs.

\begin{dfn}[{\cite[Definition 5.2.5]{larson-zapletal-GST}}]
$P$-pairs $\langle Q, \dot A \rangle$, $\langle R, \dot B \rangle$ are \emph{balance-equivalent}, denoted
$\langle Q, \dot A \rangle \balanceEq \langle R, \dot B \rangle$, if and only if for all pairs
$\langle Q', \dot A' \rangle \le \langle Q, \dot A \rangle$, $\langle R', \dot B' \rangle \le \langle R, \dot B \rangle$,
\[ Q' \times R' \Vdash \exists q \in \dot A' \ \exists r \in \dot B' \ \exists p \ p \le q, r. \]
\end{dfn}

\begin{prop}[{\cite[Proposition 5.2.6]{larson-zapletal-GST}}] \label{lem:balance-eq}
The relation $\balanceEq$ is indeed an equivalence relation,
and if $\langle Q, \dot A \rangle \le \langle R, \dot B \rangle$, then
$\langle Q, \dot A \rangle \balanceEq \langle R, \dot B \rangle$.
\end{prop}

An important property of balance equivalence is that every balance equivalence class contains a virtual condition, which is in fact
unique up to equivalence of virtual conditions, so when working with $P$-pairs up to balance equivalence it suffices to consider
virtual conditions.

\begin{prop}[{\cite[Theorem 5.2.8]{larson-zapletal-GST}}]
For any Suslin forcing $P$, every balance equivalence class of $P$-pairs contains a virtual condition which is unique up to
equivalence of virtual conditions.
\end{prop}

\begin{dfn}[{\cite[Definition 9.3.1]{larson-zapletal-GST}}]
Let $P$ be a Suslin forcing.
A virtual condition $\overline p$ of $P$ is \emph{placid} (\emph{balanced}) if and only if for all generic extensions $V[g]$, $V[h]$ such that $V[g] \cap V[h] = V$ (all pairs of mutually-generic extensions $V[g]$, $V[h]$)
 and all conditions $q \in V[g]$, $q' \in V[h]$, with $q, q' \le \overline p$,
$q$ and $q'$ are compatible.
$P$ is \emph{placid} (\emph{balanced}) if and only if for every condition $p \in P$ there is a placid (balanced) virtual condition $\overline p \le p$.
A $P$-pin is called balanced or placid when its virtual equivalence class is a balanced or placid virtual condition, respectively,
and a $P$-pair is called balanced or placid when it is balance-equivalent to a balanced or placid $P$-pin.
\end{dfn}

As we shall see in the course of the main proofs in this paper, balanced and placid pairs are of great utility in showing that specific
statements are forced, because if a statement is not decided by a balanced pair $\langle Q, \dot A \rangle$ it is often possible to
use this fact to construct incompatible pairs below $\langle Q, \dot A \rangle$.
In particular, a balanced pair for a forcing $P$ decides everything about the generic object for $P$, in the context of the Solovay model $W$, in the following sense:
\begin{thm}[{\cite[prop. 5.2.4]{larson-zapletal-GST}}]
Let $P$ be a Suslin poset and $\langle Q, \tau \rangle$ a balanced pair for $P$.
Then for any formula $\phi$ and parameter $z \in V$, one of the following holds in $V$:
\begin{itemize}
\item $Q \Vdash \Coll{\aleph_0}{<\lambda} \Vdash \tau \Vdash_P W[\dot g] \models \phi(\dot g, \check z)$;
\item $Q \Vdash \Coll{\aleph_0}{<\lambda} \Vdash \tau \Vdash_P W[\dot g] \models \neg \phi(\dot g, \check z)$,
\end{itemize}
where $\dot g = \{ \langle \check p, p \rangle : p \in P \}$ denotes the canonical $P$-name for a $P$-generic filter. 
\end{thm}

\begin{dfn}[{\cite[Definition 13.1.1]{larson-zapletal-GST}}]
A virtual condition $\overline p$ of a Suslin forcing $P$ is \textit{$(3,2)$-balanced} if and only if for every triple of filters $g_0$, $g_1$, $g_2$ for respective posets $Q_0$, $Q_1$, $Q_2$ which are pairwise mutually generic over $V$ (meaning that $g_i \times g_j$ is $Q_i \times Q_j$-generic over $V$ for $i \ne j$),
and for all conditions $p_0$, $p_1$, $p_2$ extending $\overline p$ and lying in the respective generic extensions $V[g_0]$, $V[g_1]$, $V[g_2]$, if $p_0$, $p_1$, $p_2$ are pairwise compatible then they have a common lower bound.
\end{dfn}

\begin{prop}[{\cite[Corollary of Theorem 13.2.1]{larson-zapletal-GST}}] \label{three-two-balance-all-uf-principal}
If $P$ is a $(3,2)$-balanced Suslin forcing and $g$ is $P$-generic over $W$, then in $W[g]$ there is no nonprincipal ultrafilter over $\nat$.
\end{prop}

We will prove that for certain preorders $\pbelow$, there is a natural $(3,2)$-balanced forcing associated to it. Therefore, the Solovay model $W$ can be extended to a model containing a prelinearization of $\pbelow$ while preserving $\neg \mathsf{Uf}(\omega)$.

\section{Funicular preorders}

\begin{dfn}\label{def:funicular}
A definable preorder $\pbelow$ on a Polish space $X$ is \emph{funicular} iff for every filter $g \times h$ generic for a product forcing $P \times Q$ and for every $x \in V[g]$ and $y \in V[h]$ with
$x \pbelow y$, there is $z \in V$ with $x \pbelow z \pbelow y$.
In this section, ``definable'' will mean Borel.
\end{dfn}

One can also relativize the definition of funicular to a class of forcings.
For $\mathcal P$ a class of forcings,
call a definable preorder $\pbelow$ on a Polish space $X$ \emph{$\mathcal P$-funicular} if and only if for every pair of forcings $P, Q \in \mathcal P$ and for every generic filter $g \times h$ for $P \times Q$,
if $x \in V[g]$ and $y \in V[h]$ satisfy $x \pbelow y$ in $V[g \times h]$, then there is $z \in V$ with $x \pbelow z \pbelow y$.
This relativization will be important in the goal of taking the inaccessible away, see Definition~\ref{def: cohen funicular}.

\begin{lemma}\label{lemma: exs of funicular preorders}
The following Borel preorders on Polish spaces are funicular:
\begin{enumerate}
\item \label{pareto-funicular} $(\reals^{\omega}, \le)$, where the comparison $\le$ is performed coordinatewise.
\item \label{nat-domination-funicular} $(\omega^\omega, \le)$.
\item \label{real-eventual-domination-funicular} $(\reals^\omega, \le^*)$, where $\le^*$ is domination modulo the ideal of finite subsets of $\omega$.
\item \label{nat-eventual-domination-funicular} $(\omega^\omega, \le^*)$.
\item \label{turing-funicular} $(2^\omega, \le_T)$, where $\le_T$ is Turing reducibility.
\item \label{ctble-reducibility-funicular} Generalizing the last item, any Borel preorder $(X, \le)$ with the property that
if $x, y \in X$ and $x \le y$, then either $x \in \text{OD}_y$ or $y \in \text{OD}_x$ via a formula which is upward absolute.

\end{enumerate}
\end{lemma}

\begin{proof}
It is clear that all these preorders are Borel.
Let $X$ be the domain of the analytic preorder $\pbelow$ under consideration, and fix forcing posets $P$ and $Q$, a filter $g \times h$ generic for $P \times Q$, and elements $x \in X^{V[g]}$ and $y \in X^{V[h]}$ with $x \pbelow y$.

(\ref{pareto-funicular}) Fix names $\dot x$ for $x$ in $V[g]$ and $\dot y$ for $y$ in $V[h]$, and a condition $(p, q) \in P \times Q$ such that $(p, q) \Vdash \dot x \le \dot y$.
For $n \in \omega$ define
\[ A_n = \{ s \in \rationals \mathrel{:} \exists p' \le p \ (p', q) \Vdash \check s > \dot x(n) \}, \]
\[ B_n = \{ s \in \rationals \mathrel{:} \exists q' \le q \ (p, q') \Vdash \check s < \dot y(n) \}. \]
Note that $\sup A_n \le \inf B_n$ because $(p, q) \Vdash \dot x(n) \le \dot y(n)$, so there is $z_n \in \reals^V$ with $\sup A_n \le z_n \le \inf B_n$.
Define $z \in (\reals^\omega)^V$ by $z(n) = z_n$, and observe that $(p, q) \Vdash \dot x \le \check z \le \dot y$.
Therefore $x \le z \le y$.

(\ref{nat-domination-funicular}) Given $x, y \in \omega^\omega$ from mutually generic extensions of $V$, by~(\ref{pareto-funicular}) there is $z \in (\reals^\omega)^V$ such that $x \le z \le y$.
Now simply note that $\lfloor z \rfloor \in (\omega^\omega)^V$ satisfies $x \le \lfloor z \rfloor \le y$.

(\ref{real-eventual-domination-funicular})---(\ref{nat-eventual-domination-funicular}) Upon ignoring finitely-many coordinates (a set of coordinates which exists in $V$), this becomes the same problem as (\ref{pareto-funicular}) and (\ref{nat-domination-funicular}), respectively.

(\ref{turing-funicular}) If $x \le_T y$ then $x \in V$, because there is a function $f$ recursive in $y$ with the
property that $f(y) = x$, and so $x$ and $y$ cannot be mutually-generic with respect to any nontrivial forcing.

(\ref{ctble-reducibility-funicular}) Suppose $x \le y$.
Without loss of generality (using a symmetric argument in the other case), we have that $x \in \text{OD}_y$ via a defining formula which is upward absolute.
Then any model of set theory containing $y$ also contains $x$, and so it is impossible to find nontrivial mutually generic extensions $V[g]$ and $V[h]$ such that $x \in V[g]$ and $y \in V[h]$.
Thus $\le$ is vacuously funicular.

\end{proof}

Item~(\ref{pareto-funicular}) of the preceding theorem in fact entails that the relation $\le_{\text{AP}}$, defined for $x, y \in \reals^\omega$ by $x \le_{\text{AP}} y$ if and only if there is a finite permutation $\pi$ of
$\omega$ such that $x \circ \pi \le y$ (coordinatewise), is funicular.
This is of importance to the theory of social welfare orders, where relations prelinearizing $\le_{\text{AP}}$ are called \emph{anonymous Pareto social welfare orders}.
The authors examine the applications to economics of the forcing theory of funicular preorders in the companion paper\cite{economics-paper}.

The notion of \emph{tranquility} of a preorder $\le$ strengthens that of funicularity by requiring that if $x \le y$, $x \in V[g]$, $y \in V[h]$, and
$V[g] \cap V[H] = V$, then there is $z \in V$ with $x \le z \le y$. This was used by the second author in~\cite{serafin-SEA} to construct a model of set theory containing a prelinearization of a preorder of importance in social choice theory
and no nonprincipal ultrafilter on $\omega$.
However, the same arguments do not work for the economically more interesting case of AP orders, since it turns out that
the relevant preorder (which extends coordinatewise $\le$ on $\RR^\omega$) is not tranquil.

\begin{prop} \label{prop: coordinatewise is not tranquil}
The relation $\le$ on $\RR^\omega$, interpreted coordinatewise, is not tranquil.
\end{prop}

\begin{proof}
    Consider the poset
    \[ P = \{ (p, q) \in \bigcup_{n < \omega} (\zahlen^n \times \zahlen^n) : \forall k < n. p(k) < q(k) \}, \]
    ordered by end-extension in both coordinates.
    This is clearly an atomless, countable poset, and so is forcing-equivalent to Cohen forcing by Theorem~\ref{thm:Cohen-properties}.
    Fix a filter $h$ which is $P$-generic over $V$,
    and let $\gamma = \bigcup_{(p, q) \in g} p$ and $\gamma' = \bigcup_{(p, q) \in g} q$.
    Take $g$ and $g'$ to be the filters generated by initial segments of $\gamma$ and $\gamma'$, respectively.
    Clearly $\gamma \le \gamma'$; we shall show that $V[\gamma] \cap V[\gamma'] = V$ and that there is no
    $z \in V$ with $\gamma \le z \le \gamma'$.

    For the latter statement it suffices to show that $\gamma$ and $\gamma'$ are $\zahlen^{<\omega}$-generic over $V$,
    since Cohen reals are unbounded~\cite[Lemma 22.1]{halbeisen-combinatorial}.
    But this is immediate, since for any dense set $D \subseteq \zahlen^{<\omega}$ it is clear that both
    $(D \times \zahlen^{<\omega}) \cap P$ and $(\zahlen^{<\omega} \times D) \cap P$ are dense in $P$.

    Now it remains to verify that $V[\gamma] \cap V[\gamma'] = V$.
    By the zero-extension of a condition $p \in \zahlen^{<\omega}$ to some length $\ell > |p|$ we mean the function
    $p' : \ell \rightarrow \zahlen$ defined by $p'(k) = p(k)$ for $k < |p|$ and $p'(k) = 0$ otherwise.
    Take $\dot g$, $\dot g'$ to be the canonical names for the left and right projections of $h$.
    Let $x \in V[\gamma] \cap V[\gamma']$, and fix names $\dot x$, $\dot y$ thereof and a condition $(p, q) \in h$
    such that $(p, q) \Vdash \dot x_{\dot g} = \dot y_{\dot g'}$.
    By $\in$-induction we may assume that $\zahlen^{<\omega} \Vdash \dot x, \dot y \subseteq V$.
    Hence assuming for contradiction that $x \notin V$,
    there are $z \in V$ and conditions $p', p'' \le p$ of the same length such that
    $p' \Vdash \check z \in \dot x_{\dot g}$ and $p'' \Vdash \check z \notin \dot x_{\dot g}$, for otherwise $p$ would decide all
    elements of $\dot x_{\dot g}$ and consequently $x$ would be in $V$.
    Now, if $q$ were to decide whether $\check z \in \dot y$, say by deciding $\check z \notin \dot y$,
    then we would have a contradiction because,
    taking $q_0$ to be an extension of $q$ to the common length of $p'$ and $p''$
    and satisfying $q_0(k) < p'(k), p''(k)$ for each $k$,
    the condition $(p', q_0)$ forces that $\check z \in \dot x_{\dot g}$, while $(p'', q_0)$ forces the opposite,
    contradicting the fact that $(p, q) \Vdash \dot x_{\dot g} = \dot y_{\dot g'}$
    since $(p', q_0), (p'', q_0) \le (p, q)$.
    Hence there are $q', q'' \le q$ of the same length such that
    $q' \Vdash \check z \in \dot y_{\dot g'}$ and $q'' \Vdash \check z \notin \dot y_{\dot g'}$.
    Extending as necessary, we may assume that $|p'| = |p''| = |q'| = |q''|$.
    Now we have a new contradiction, because
    \[ (p', q'') \Vdash \check z \in \dot x_{\dot g} \setminus \dot y_{\dot g'}, \]
    and yet $(p', q'') \le (p, q)$ and $(p, q) \Vdash \dot x_{\dot g} = \dot y_{\dot g'}$.
    \end{proof}

An example of a Borel preorder which is not funicular was suggested to us by Jonathan Schilhan.
Consider the clopen subsets $[0]$ and $[1]$ of $2^\nat$ which consist of those sequences beginning with $0$ and those beginning with $1$, respectively.
Order these by $x \precsim y$ if and only if either $x = y$ or $x \in [0]$ and $y \in [1]$.
This relation is obviously Borel, and it is not funicular because if $g, g'$ are mutually-generic Cohen reals with $g(0) = 0$ and $g'(0) = 1$, then $g$ and $g'$ witness a failure of funicularity.

\section{Forcing prelinearizations over the Solovay model} \label{sec:solovay-model-prelinearization}

\begin{dfn}
For $\pbelow$ an analytic preorder on a Suslin space $X$, the forcing $\PrelinPoset[\pbelow]$ consists of enumerations of countable $R \subseteq X \times X$ satisfying that $R \restriction \dom(R)$ is a prelinearization of
$\pbelow \restriction \dom(R)$, ordered by reverse end-extension of the enumerated relations.
\end{dfn}

\begin{lemma} \label{prelin-extend}
For $\pbelow$ a Borel preorder on a Polish space $X$, $\PrelinPoset[\pbelow]$ is a Suslin $\sigma$-closed forcing.
\end{lemma}

\begin{proof}
Because $X$ is a Polish space so is $(X \times X)^\omega$, and $\PrelinPoset[\pbelow] \subseteq (X \times X)^\omega$.
We write $\PrelinPoset$ for $\PrelinPoset[\pbelow]$ for the remainder of the proof.
The condition that $R \in \PrelinPoset$ is prelinear is clearly Borel, because $R$ is countable and the definition of a prelinear order is first-order.
Apply a Borel function to the enumeration of a given $R \in \PrelinPoset$ to produce an enumeration $\langle x_n \rangle_{n < \omega}$ of $\dom(R)$,
for example by enumerating $\dom(R)$ in the same order as $R \cap \Delta$ where $\Delta \subseteq X \times X$ is the diagonal.
The prelinearization condition requires that
\[ \forall i, j \ (x_i \pbelow x_j \wedge x_j \npbelow x_i) \Rightarrow (x_i R x_j \wedge \neg x_j {R} x_i), \]
which is Borel because the complement of a Borel set is Borel.
(Note this argument would fail were $\pbelow$ allowed to be an arbitrary analytic preorder.)
Reverse inclusion of enumerations is clearly a Borel relation, and therefore $\PrelinPoset$ is a Suslin forcing. 

Now suppose $\langle R_n \rangle_{n < \omega}$ is an end-extension-increasing sequence of conditions in $\PrelinPoset$, and let $R$ be the enumeration of $\bigcup_{n < \omega} R_n$ in order of appearance of each pair.
$R$ clearly enumerates a prelinear order.
Now suppose $x, y \in \dom(R)$, $x \pbelow y$, and $y \npbelow x$.
Choose $n$ such that $x, y \in R_n$.
Then because $R_n \in \PrelinPoset$, $x R y$ and $\neg (y \mathrel{R} x)$.
Since $R \in \PrelinPoset$ and $R$  end-extends $R_n$ for each $n$, we conclude that $\PrelinPoset$ is $\sigma$-closed.
\end{proof}

Note that $\sigma$-closed forcing over a model of $\ZF + \DC$ preserves $\DC$, a folklore result. 

\begin{lemma}\label{lemma: funicular implies 32balanced}
Let $\pbelow$ be a funicular Borel preorder on a Polish space $X$.
\begin{enumerate}
\item \label{total-balance} For $\pless$ a prelinearization of $\pbelow$ in $V$, $\langle P, \check{\pless} \rangle$ is a $(3,2)$-balanced virtual condition, where $P = \Coll{\omega}{X}$. 
\item \label{total-classifies} If $\langle Q, \dot R \rangle$ is a $(3,2)$-balanced virtual condition, then there is in $V$ a total prelinearization $\pless$ of $\pbelow$ such that
$\langle Q, \dot R \rangle$ is balance-equivalent to $\langle P, \check{\pless} \rangle$.
\item \label{prelin-balanced} The poset $\PrelinPoset[\pbelow]$ is $(3,2)$-balanced.
\end{enumerate}
\end{lemma}

\begin{proof}
(\ref{total-balance}) Working in $V$, let $\pless$ be a prelinearization of $\pbelow$, which exists by Lemma~\ref{prelin-extend}.
Fix pairwise mutually-generic filters $\{g_k \mathrel{:} k < 3\}$ and conditions $R_k \in V[g_k]$ with each $R_k \in \PrelinPoset[\pbelow]^{V[g_k]}$ (which makes sense because $\PrelinPoset[\pbelow]$ is Suslin)
and $\pless \subseteq R_k$.
Let $V[K]$ be a generic extension with $g_k \in V[K]$ for $k < 3$ and work now in $V[K]$.
If the conditions $\{ R_k \mathrel{:} k < 3 \}$ have no common lower bound then there is a directed cycle in the relation $R^s_0 \cup R^s_1 \cup R^s_2 \cup \psbelow$,
where for $k < 3$ $R^s_k = R^s_k \cap \neg \breve{R^s_k}$ is the strict form of $R_k$ and $\psbelow$ is the strict form of $\pbelow$.
Note that by transitivity of the component relations we may assume that no two successive edges of the directed cycle satisfy the same relation among $\{ R^s_0, R^s_1, R^s_2, \psbelow \}$.

We proceed by induction on the number of occurrences of the relation $\psbelow$ in the cycle.
If $\psbelow$ does not occur in the cycle, then because the same relation among the $R^s_k$'s does not occur twice consecutively,
every vertex of the cycle is in the intersection of two distinct models of the form $V[g_k]$ and hence by mutual genericity is in $V$.
But each strict prelinearization $R^s_k$ agrees with $\psless$ for all pairs of elements of $V$, so there is a cycle in $\psless$, the strict version of $\pless$, which is impossible.
Now suppose that there are no cycles with $n$ occurrences of $\psbelow$, and suppose we have a cycle $C$ with $n+1$ occurrences of $\psbelow$.
Let $x \psbelow y$ be an occurrence of $\psbelow$ in $C$.
Because $\pbelow$ is funicular, there is $z \in V$ with $x \pbelow z \pbelow y$
Suppose $x \in V[g_k]$ and $y \in V[g_i]$; then $x R_k z R_i y$, and at least one of these relations is strict.
If one of them is non-strict we can remove the point which isn't $z$ from the cycle.
In any case we obtain a cycle with $n$ occurrences of $\psbelow$, and by induction there is a cycle with all edges in $\psless$, which is impossible.

(\ref{total-classifies}) Suppose $\langle Q, \dot R \rangle$ is $(3,2)$-balanced.
Then for every pair $x, y \in X$, either $Q \Vdash \check x \mathrel{\dot R} \check y$ or $Q \Vdash \check y \mathrel{\dot R} \check x$,
as otherwise there would be incompatible conditions in mutually generic extensions realizing $\dot R$.
Let $\pless$ be the relation on $X^V$ determined by $\dot R$, and note that this is a prelinearization of $\pbelow$ in $V$.
Because $Q$ forces that $\dot R \in \PrelinPoset[\pbelow]$ is countable, $\langle \Coll{\omega}{X}, \pless \rangle \le \langle Q, \dot R \rangle$, and therefore these are balance-equivalent by lemma~\ref{lem:balance-eq}.

(\ref{prelin-balanced}) What remains is to show that for every $p \in \PrelinPoset[\pbelow]$ there is a $(3,2)$-balanced virtual condition $\overline p \le p$.
Fixing $p$, by Lemma~\ref{prelin-extend} there is a total prelinearization $\pless$ extending $p$, and then $\langle \Coll{\omega}{X}, \pless \rangle \le p$ is $(3,2)$-balanced.
\end{proof}

\begin{cor}
Let $\pbelow$ be a Borel funicular preorder.
If $\ZF$ plus the existence of an inaccessible cardinal is consistent, then so is the theory
\[ \ZF + \text{``there is a prelinearization of $\pbelow$''} + \text{``all ultrafilters over $\omega$ are principal''}. \]
\end{cor}

\begin{proof}
Fix an inaccessible cardinal $\kappa$, and let $\tilde V$ be a generic extension of $V$ by the forcing $\Coll{\aleph_0}{<\kappa}$ and $W \subseteq \tilde V$ be the Solovay model derived from $\kappa$.
Consider the model $W[g]$, where $g$ is generic over $\tilde V$ (and hence also generic over $W$) for $\PrelinPoset = \PrelinPoset[\pbelow]$.
Certainly $W[g] \models \ZF$ because $W \models \ZF$.
A straightforward genericity argument shows that in $W[g]$, $\cup g$ is a prelinearization of $\pbelow$.
The fact that in $W[g]$ all ultrafilters over $\omega$ are principal follows from a combination of the preceding lemma, which established that $\PrelinPoset$ is $(3,2)$-balanced, and Proposition~\ref{three-two-balance-all-uf-principal},
which states that in $(3,2)$-balanced generic extensions of $W$, all ultrafilters over $\omega$ are principal.
\end{proof}

\section{Eliminating the inaccessible} \label{section: no inaccessible}
The goal of this section is to reproduce the arguments in the last sections in a different context---one in which the inaccessible is not needed. The advantage of geometric set theory approach is that we can use machinery already constructed, and this part will be more \textit{ad hoc} while working towards a stronger result.

The models in this section will be inner models of $V[g][h]$ where $g$ is a $\mathbf{Q}$-generic filter over $V$, and $h$ is a $\mathbf{P}$-generic filter over $V[g]$.
In this section $\mathbf{Q}$ denotes $\mathbf{C}(\omega_1)$ and $g$ denotes a $\mathbf{Q}$-generic filter over G\"odel's constructible universe $L$.
Mostly $\mathbf{P}$ will be a forcing notion approximating the prelinearization considered, which we will denote $\Pprelin$.

\subsection{Free ultrafilters and $(m,n)$-amalgamation.}

\begin{dfn}
	Let $\mathbf{P}$ be a forcing poset in $L[g]$ of the form 
	\[p \in \mathbf{P} \iff \exists x \in \RR\; L[x]\models \text{``$p\subseteq \RR$ and $\psi(p)$''}, \]
	ordered by reverse inclusion, and such that for every $p$, any real witness $x$ for being a condition is definable from a finite set of elements of $p$. 
	Then we say that $\mathbf{P}$ is a \emph{local forcing}.
\end{dfn}

Notice that the definability condition implies that ``$p\in \mathbf{P}$'' is absolute between inner models and not only upwards absolute. 
Recall as well that for every real $x$ in $L[g] $ there is a $\gamma<\omega_1$ such that $x\in L[g\restriction\gamma]$ (see Lemma~\ref{lemma: prereq any real is in an inital segment of C(omega1)}), and that for every $\gamma<\omega_1$ there is a real $x$ such that $L[x]=L[g\restriction \gamma]$.

\begin{dfn}
    Let $\mathbf{P}$ be a local forcing in $L[g]$ and $n,m\in \omega$ with $n\geq m>0$.
	We say that $\mathbf{P}$ satisfies \emph{$(n,m)$-amalgamation} if and only if for all $\beta < \omega_1$ there is a dense set of $p\in \mathbf{P}$ such that 
	\begin{enumerate}[i.]
		\item there is some $\gamma$, with $\beta\leq \gamma < \omega_1$, such that $L[g\restriction \gamma]\models p \subseteq \RR \text{ and } \psi(p)$, \label{item: def - extendability}
		and 
		\item \label{item: amalgamation th ultrafilter} for all $g_0, \dots, g_{n-1}$ $\mathbf{Q}$-generic filters over $L[g\restriction\gamma]$ such that any $m$ of these are mutually generic, and $\bigcap_{i\in n} L[g_i]=L[g\restriction\gamma]$; and 
	  for every set of conditions $\{p_i\}_{i\in n}$ extending $p$ such that $p_i \in \mathbf{P}^{L[g_i]}$ for every $i\in n$, the conditions $\{p_i\}_{i<n}$ are compatible in $V[g]$. 
	\end{enumerate}
\end{dfn}

Notice that the condition $\bigcap_{i\in n} L[g_i]=L[g\restriction\gamma]$ is redundant when $m>1$ since it follows from $m$-mutual genericity.  

Notice the similarities of $(n,m)$-amalgamation with the concepts described in the past sections: for $n=m=2$ it is similar to \emph{balanced}, for $n=3, m=2$ it is like \emph{$(3,2)$-balanced}, $n=2, m=1$ is analogous to \emph{placid}. 
Also, a version of amalgamation which corresponds roughly to the $(2,2)$-amalgamation defined here is explored in~\cite[Definition 3.10 \& Theorem 3.11]{fatalini-unit-circle-partition} as a sufficient criterion for extensions not adding any wellorder of the reals.
The usefulness of amalgamation to us is to guarantee that models contain no nonprincipal
ultrafilter over $\omega$; see Theorem~\ref{th: no ultrafilter no inaccessible}.

Observe that $(n,m)$-amalgamation implies $(n,m')$-amalgamation for $n \ge m'\geq m$. 

\begin{thm} \label{th: no ultrafilter no inaccessible} 
	Let $\mathbf{P}$ be a local forcing poset in $L[g]$ that satisfies $(n,n-1)$-amalgamation for some $n>1$ and is $\sigma$-closed.
	Let $h$ be a $\mathbf{P}$-generic filter over $L[g]$ and $\mathcal{P}=\cup h$. 
	Then 
	\[L(\RR, \mathcal{P})^{L[g,h]}  \models \mathsf{ZF} + \mathsf{DC}+\neg\Uf(\omega). \]
\end{thm}

\begin{proof}
Notice that $L(\RR, \mathcal{P})^{L[g,h]}  \models \mathsf{ZF}+\mathsf{DC}$ because $\mathbf{P}$ is $\sigma$-closed. 
We only need to show that $L(\RR, \mathcal{P})^{V[g,h]}  \models \neg\Uf(\omega)$.

	\textbf{Step I:} Suppose the contrary.
	Then there is some formula $\phi(\cdot, \vec{x}, \vec{\alpha}, \mathcal{P})$, where $\vec{x} \subseteq \RR^{V[g,h]}$ and $\vec{\alpha}\subseteq \OR$, such that
\begin{equation} \label{eq: ultrafilter step1}
	L[g,h] \models \phi \left(\cdot, \vec{x}, \vec{\alpha}, \mathcal{P}\right) \text{ defines a non principal ultrafilter on } \omega.
\end{equation}
Since $\mathbf{P}$ is $\sigma$-closed, $\RR^{L[g,h]}=\RR^{L[g]}$. 
Recall that any real in $L[g]$ is in an inital segment of $g$ (see Lemma~\ref{lemma: prereq any real is in an inital segment of C(omega1)}), that is, there is $\beta<\omega_1$ such that $\vec{x}\in \RR^{L[g\restriction\beta]}$, and we can assume $\beta > \max\vec{\alpha}$.

There is  a condition $p$ that forces the statement \eqref{eq: ultrafilter step1}, i.e.
\begin{equation}\label{eq: ultrafilter step2}
	p \forc{\mathbf{P}}{L[g]} \phi \left(\cdot, \check{\vec{x}}, \check{\vec{\alpha}},\dot{\mathcal{P}}\right) \text{ defines a non principal ultrafilter on } \omega, 
\end{equation}
where $\dot{\mathcal{P}}$ is a definable name for $\set{ \tup{\check{x},p} \mid x\in p \text{ and }p\in \mathbf{P}}$. 
Henceforth every time we write $\mathbf{P}$ we are thinking about its explicit formula defining it as a local forcing, which as it was noticed has the property that belonging to $\mathbf{P}$ is absolute bewteen inner models.

By hypothesis we can assume that $p$ has the properties described in the definition of $(n+1,n)$-amalgamation, in particular,  $p \in L[g\restriction	\gamma]$ for some $\gamma > \beta$. 
Consider the forcing $\mathbf{Q}$ as adding $\omega_1$ Cohen functions from $\omega$ to $2$.

Now consider the $\gamma^{th}$ Cohen real defined by $g$ and call it $r$.
Notice that $r\colon \omega \to 2$, so 
\[ r^{-1}(0)  \dot{\cup} r^{-1}(1)= \omega. \]

Therefore, in $L[g,h]$ one of these sets, call it $\bar{r}$, must be in the ultrafilter defined by $\phi$. 
Without loss of generality assume $\bar{r} = r^{-1}(0)$.
Pick $\bar{p}\leq p$ such that $\bar{p}$ forces that $\bar{r}\subseteq \omega$ is in the ultrafilter, which is to say, 

\begin{equation}\label{eq: ultrafilter step3}
	\bar{p} \forc{\mathbf{P}}{L[g]} \phi \left(  \check{\bar{r}}, \check{\vec{x}}, \check{\vec{\alpha}},\dot{\mathcal{P}}\right).
\end{equation}

By definability of forcing there is some formula that expresses the forcing relation described in \eqref{eq: ultrafilter step3} in $L[g]$.  
Now write $\mathbf{Q}$ as $\mathbf{C}(\gamma) \times \mathbf{C}(\omega_1\backslash \gamma)$ and think of $g$ as
\[g= (g\restriction \gamma) \times g \restriction (\omega_1 \backslash \gamma).\]
Then there is a condition $q\in \mathbf{C}(\omega_1 \backslash \gamma)$ such that 

\begin{equation} \label{eq: ultrafilter step4}
	q \forc{\mathbf{C}(\omega_1\backslash \gamma)}{L[g\restriction \gamma]}  ``
	 \dot{\bar{p}}\leq_{P} \check{p} \text{ and } 
	\dot{\bar{p}} \forc{\mathbf{P}}{L[g\restriction\gamma, \dot g \restriction (\omega_1 \setminus \gamma)]} \phi \left(  \dot{\check{\bar{r}}}, \check{\check{\vec{x}}}, \check{\check{\vec{\alpha}}},\dot{\mathcal{P}} \right).'' 
\end{equation}
where $\dot{\bar{p}}$ is a $\mathbf{C}(\omega_1\backslash \gamma)$-name for $\bar{p}$ and $\dot{\check{\bar{r}}}$ the following $\mathbf{C}(\omega_1\backslash \gamma)$-name for $\check{\bar{r}}$: 

\begin{equation} \label{eq:def-r-bar-check-dot}
\dot{\check{\bar{r}}} = \left\{ 
\tupp{\tupp{\check{k}, \text{$\mathbbm{1}$} _{\mathbf{P}}}^{\vee}, t} \mid t\in \mathbf{C}(\omega_1\backslash \gamma), \tupp{k, 0}\in t(\gamma) \right\}.
\end{equation}

Notice that $\gamma > \beta$ implies that $\vec{\alpha}, \vec{x}\in L[g\restriction\beta] \subseteq L[g\restriction\gamma]$, and recall that $p\in L[g\restriction \gamma]$. \\

\textbf{Step II:} 
We want to get a contradiction from \eqref{eq: ultrafilter step4}. 
The idea is to extend $L[g\restriction \gamma]$ in different (but still amalgamable) ways so that we get $n$ sets which are in the corresponding ultrafilter, but whose intersection is finite, a contradiction.
This step consists in carefully constructing such extensions. 

Let $s$ be $q(\gamma)\in 2^{<\omega}$.
Consider the forcing $\mathbf{C}_s$ in $L[g\restriction \gamma]$ given by 
\begin{align*}
	(a_0\dots,a_{n-1})\in \mathbf{C}_s \iff &  a_0, \dots, a_{n-1} \in \mathbf{C} \text{ with }  a_0, \dots, a_{n-1}  \supseteq s,\\
	& \lh(a_0)=\dots=\lh(a_{n-1}); \text{ and } \\
	& \{l \geq \lh(s) \mid a_0(l)=\dots=a_{n-1}(l)=0\}= \emptyset,
\end{align*}
ordered by reverse inclusion in each coordinate. 

\begin{claim}
	The following hold:
	\begin{enumerate}[i.]
		\item The definition of $\mathbf{C}_s$ is absolute.
		\item $\mathbf{C}_s\cong \mathbf{C}$. 
		\item If $g_s$ is a $\mathbf{C}_s$-generic filter over some transitive model $M$, consider for $i\in n$ the projection $j_i$ of $g_s$ to the $i^{\text{th}}$ coordinate. Then $j_i$ is $\mathbf{C}$-generic over $M$ and $M[g_s]$ is a $\mathbf{C}$-extension of $M[j_i]$ for every $i\in n$.
		\item Every subset of $\{j_i\}_{i\in n}$ of size $n-1$ is a set of mutually generic filters over $M$. \label{item: n-1 mutual genericity}
		\item  $\bigcap_{i\in n} M[j_i]=M$. \label{item: coord of C_s are almost mut ivgeneric}
		\item  If $r_i$ is the Cohen real defined by $j_i$ for each $i\in n$, then $\{l \mid r_0(l)=\dots=r_{n-1}(l)=0 \}$ is finite. 
	\end{enumerate}
\end{claim}

\begin{claimproof}
	\begin{enumerate}[i.]
		\item It is clear. 
		\item Because $\mathbf{C}_s$ is countable and without atoms, $\mathbf{C}_s\cong \mathbf{C}$.
		\item For each $i \in \{1,2\}$, $j_i\subseteq M$ and $j_i\in M[g_s]$, so by
	Solovay's Basis Theorem \cite[Theorem~2 Section~2.14]{Grigorieff1975} we obtain that $M[j_i]$ is a forcing extension of $M$ and  $M[g_s]$ is a forcing extension of $M[j_i]$.
		
		Moreover, consider $r_i=\cup j_i$ for $i\in n$.
		Notice $r_i$ is a real and $M[r_i]=M[j_i]$.
		We know also that $\mathbf{C}_s\cong \mathbf{C}$ so we can think of $M[g_s]$ as a Cohen extension of $M$. 
		Then $r_i$ is a real in a Cohen extension which means $M[r_i]$ is itself a Cohen extension of $M$ and the forcings corresponding to $j_0$ and $j_1$ have to be forcing equivalent to $\mathbf{C}$ (see Theorem \ref{thm:Cohen-properties}~\ref{th: prereq any real in cohen ext is cohen}). 
		
		\item By the symmetry of the coordinates,  it is enough to show that $\{j_i\}_{i\in n-1}$ are mutually generic filters over $M$.
		In other words, we need to prove that $j_0\times\dots\times j_{n-2}$ intersects every dense set $D$ on $\mathbf{C}^n$.  
		Fix such $D$ and define $D'\subseteq \mathbf{C}_s$ as follows. 
		
		\[D'=\{(a_i)_{i\in n}\in \mathbf{C}_s \,|\, (a_i)_{i\in n-1}\in D \}.\]
		
		We claim that $D'$ is dense in $\mathbf{C}_s$. 
		Fix $\vec{b}\in \mathbf{C}_s$ and denote $\vec{b} \restriction (n - 1)$ by $\vec{b}^\ast$.
		By density of $D$, there is $\vec{a}\supseteq \vec{b}^\ast$ such that $\vec{a}\in D$ and we can assume that all the coordinates in $a$ have the same length $l$. 
		Now, let $p\supseteq b_{n-1}$ be a condition in $\mathbf{C}$ of length $l$ such that for all $j$ satisfying $\lh(s)\leq j < l$, if $a_0(j)=\dots=a_{n-2}(j)=0$ then $p(j)=1$. 
		Then clearly $(\vec{a},p)\in \mathbf{C}_s$ and moreover $(\vec{a},p)\in D'$.
		
		Since $D'$ is dense, there is $\vec{c}\in g_s$ such that $\vec{c}\in g_s\cap D'$. Then by definition of $D'$, $\vec{c}^\ast\in (j_0\times\dots\times j_{n-2}) \cap D$, as we wanted.		
		
		\item For $n>2$, this is clear due to the $(n-1)$-mutual-genericity shown in item \ref{item: n-1 mutual genericity}. If $n=2$, then we need to prove that $M[j_0]\cap M[j_1]=M$. 
		Suppose that $y\in M[j_0]\cap M[j_1]$.
		Then there are $\mathbf{C}$-names $\sigma$ and $\tau$ such that $y=\sigma_{j_0}=\tau_{j_1}$. 
		Let us assume by $\in$-induction that $y\subseteq  M$. 
		Since $y\in M[g_s]$, there is a condition $(b,c)\in \mathbf{C}_s$ such that 
		\[(b,c) \forc{\mathbf{C}_s}{M} \tilde{\sigma} = \tilde{\tau} \text{ and } \tilde{\sigma} \subseteq M, \]
		where $\tilde{\sigma}$ and $\tilde{\tau}$ are the $\mathbf{C}_s$-names that resemble $\sigma$ and $\tau$ respectively but each name contains every condition in the second, respectively first, coordinate.
		We will describe $\tilde{\sigma}$ and $\tilde{\tau}$ precisely as follows. 
		
		Given a name $\sigma \in M^\mathbf{C}$, we define $\sigma^{\ast}$ and $^\ast \sigma$ in $M^{\mathbf{C}_s}$ recursively:
		\begin{align*}
			\sigma^\ast &= \{(\pi^\ast,(a_0,a_1)) \mid (a_0,a_1)\in \mathbf{C}_s \text{ and }(\pi, a_0)\in \sigma\}  \\
			^\ast\sigma &= \{({^\ast}\pi,(a_0,a_1)) \mid (a_0,a_1)\in \mathbf{C}_s \text{ and }(\pi, a_1)\in \sigma\}
		\end{align*}
		Set $\tilde{\sigma}= \sigma^\ast$ and $\tilde{\tau}={}^{\ast}\tau$. 
		
		We claim that $b$ decides ``$\check{z}\in \sigma$'' for every $z\in M$. 
		If this holds, then $y=\sigma_{j_1}\in M$.
	
		Suppose it does not. 
		Then there is some $z\in M$ for which $b$ does not decide whether $\check{z}$ belongs to $\sigma$.
		Pick $b_0$, $b_1\in \mathbf{C}$ extending $b$ such that 
		\begin{align*}
			b_0 & \forc{\mathbf{C}}{M}  \check{z} \not\in \sigma,  \text{ and} \\
			b_1 & \forc{\mathbf{C}}{M} \check{z} \in \sigma.
		\end{align*}				
		Let $c'$ be the extension of $c$ to length $\max \{\lh(b_0), \lh(b_1)\}$ such that $c'=c{}^\frown \vec{1}$. 
		Extend $c'$ to $c''$ so that $c'$ decides ``$\check{z}\in \tau$''.
		Let us assume without loss of generality that $c''\forc{\mathbf{C}}{M} \check{z}\not \in \tau$.
		Then let $b'$ be the extension of $b_1$ of length $\lh(c'')\geq \lh(b_1)$ such that $b'=b_1{}^\frown \vec{1}$, as Figure \ref{fig: ultrafilter claim Cs steps} shows.
		If $c'' \forc{\mathbf{C}}{M} \check{z} \in \tau$, we would have taken $b'=b_0{}^\frown \vec{1}$.
		By construction, $(b', c'')\in \mathbf{C}_s$ and $(b', c'')\forc{\mathbf{C}_s}{M} \check{z} \in \tilde{\sigma} \wedge \check{z} \not\in \tilde{\tau}$.
        
		\begin{figure}[H] 
		    \centering
		    \includegraphics[width=0.7\linewidth]{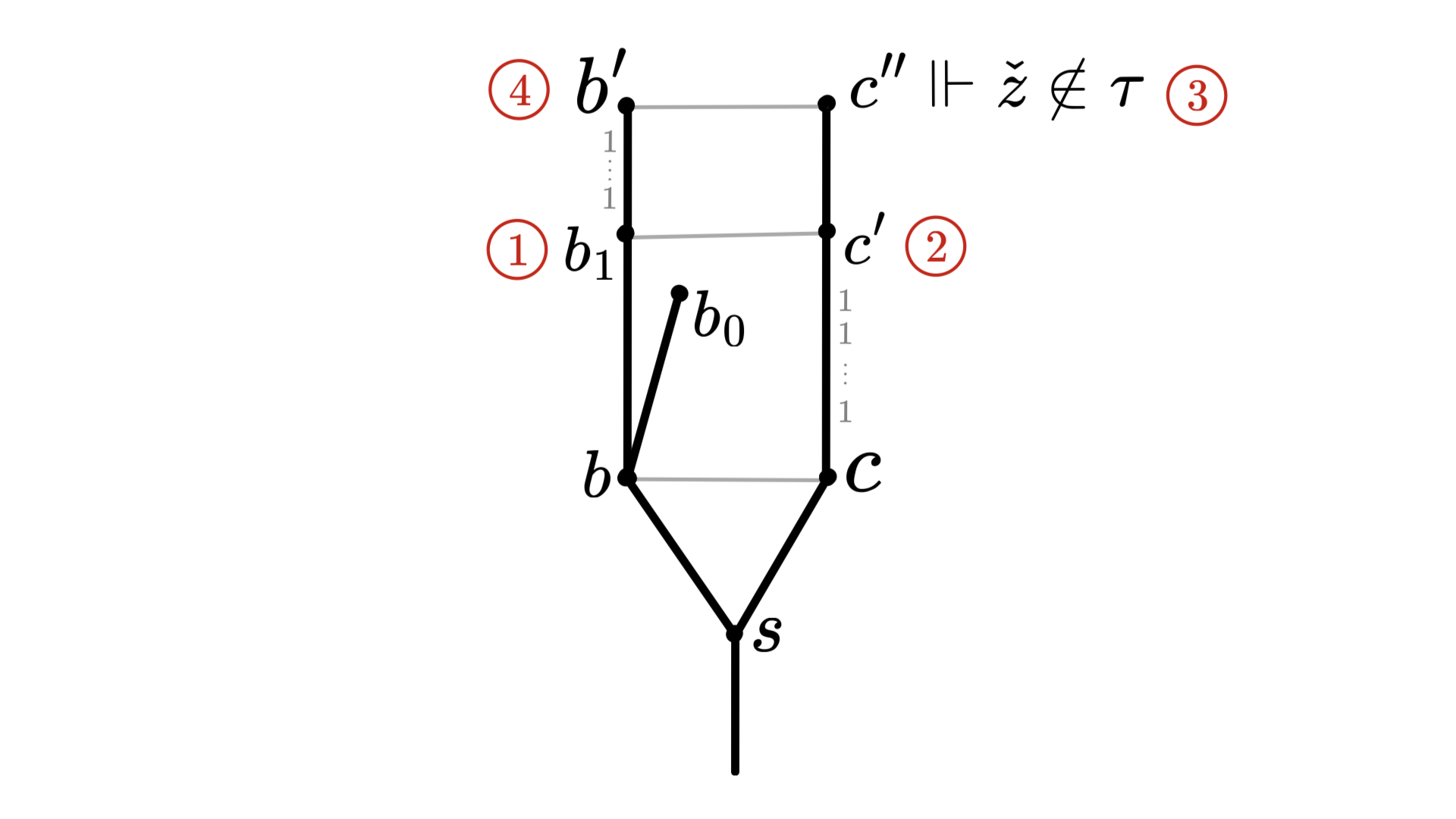}
		    \caption{$(b', c'')\in \mathbf{C}_s$, but $b'\Vdash\check{z}\in \sigma$ and $c''\Vdash \check{z}\not \in \tau$.}
		    \label{fig: ultrafilter claim Cs steps}
		\end{figure}
	
		This is a contradiction, since $(b',c'')\leq (b,c)$. 
		\item 	Let $r_i$ be the real added by $j_i$ for each $i\in n$. 
		Then  $\{n \mid r_0(l)=\dots=r_{n-1}(l)=0\}=\{l \mid s(l)=0\}$, which is finite.	
	\end{enumerate}
\end{claimproof}

\textbf{Step III:}
Let $g_s$ be a $\mathbf{C}_s$-generic filter over $L[g\restriction \gamma]$ and let $\tilde{g}_0,\dots,\tilde{g}_{n-1}$ be $n$ $\mathbf{C}(\omega_1\backslash (\gamma+1))$-mutually generic filters over $L[g\restriction \gamma, g_s]$, each containing $q\restriction(\omega_1\backslash (\gamma+1))$.
This can be done so that $L[g\restriction\gamma, g_s, \tilde{g}_0,\dots,\tilde{g}_{n-1}] \subseteq L[g]$.
For each $i\in n$ we have that $\tilde{g}_i$ is also $\mathbf{C}(\omega_1\backslash (\gamma+1))$-generic over $L[g\restriction \gamma, j_i]$. 
In other words, $j_i \times \tilde{g}_i$ is $\mathbf{C}(\omega_1\backslash \gamma)$-generic over $L[g\restriction \gamma]$ and also $q\in j_i \times \tilde{g}_i$ for all $i\in n$.

We have now our $n$ different (but amalgamable) paths: \\

\begin{figure}[H] \centering
	\begin{tikzcd}
		& &L[g\restriction \gamma][j_0][\tilde{g}_0] \arrow[rrd, hookrightarrow, end anchor= real west, start anchor=base east, "\subseteq", sloped]&& \\
		L[g\restriction \gamma]  \arrow[rru, hookrightarrow, end anchor=base west, start anchor=east, "\subseteq", sloped] \arrow[rrd, hookrightarrow, end anchor=base west, start anchor=base east, "\subseteq", sloped]&&\vdots&& L[g\restriction \gamma][g_s][\tilde{g}_0,\dots, \tilde{g}_{n-1}] \\
		&&L[g\restriction \gamma][j_{n-1}][\tilde{g}_{n-1}] \arrow[rru, hookrightarrow, end anchor=base west, start anchor=base east, "\subseteq", sloped]&&
	\end{tikzcd}
\end{figure}

For $i\in n$ let $g_i$ be $(g\restriction \gamma) \times j_i \times \tilde{g}_i$. 

Now, recall \eqref{eq: ultrafilter step4}: 
\begin{equation*}
	q \forc{\mathbf{C}(\omega_1\backslash \gamma)}{L[g\restriction \gamma]}  
\text{``$\dot{\bar{p}}\leq_{P} \check{p}$ and 
	$\dot{\bar{p}} \forc{\mathbf{P}}{L[g\restriction\gamma, \cdot]} \phi \left(  \dot{\check{\bar{r}}}, \check{\check{\vec{x}}}, \check{\check{\vec{\alpha}}},\dot{\mathcal{P}} \right)$.''}
\end{equation*}

Here $\dot{\check{\bar{r}}}$ is as in~\eqref{eq:def-r-bar-check-dot}.
Let $p_i= (\dot{\bar{p}})_{j_i\times \tilde{g_i}}$ for all $i\in n$.
Looking at the definition of $\dot{\check{\bar{r}}}$ and computing $(\dot{\check{\bar{r}}})_{g_i}$, we know that $(\dot{\check{\bar{r}}})_{j_i\times \tilde{g}_i}$ is just the check name of the preimage of $0$ of the real $r_i$ in $L[g_i]=L[g\restriction \gamma, j_i, \tilde{g}_i]$.
Let us then write $\check{\bar{r}}_i=(\dot{\check{\bar{r}}})_{j_i\times \tilde{g_i}}$.
Then we get that $p_i\leq_\mathbf{P} p$ in $L[g]$ (and in any other model that has $p$ and $ p_i$ as elements, since this relation is absolute).
Moreover, for each $i\in n$
\begin{equation}
	p_i \forc{\mathbf{P}}{L[g_i]} \phi \left(  \check{\bar{r}}_i, \check{\vec{x}}, \check{\vec{\alpha}},\dot{\mathcal{P}}\right). \label{eq: ultrafilter step5}
\end{equation}

Notice that $\bigcap_{i\in n}L[g_i]= L[g\restriction\gamma]$ using the item \ref{item: coord of C_s are almost mut ivgeneric} of the claim and mutual genericity of $\{\tilde{g}_i\}_{i\in n}$. 
The hypothesis of $(n,n-1)$-amalgamation now implies that the conditions $\{p_i\}_{i\in n}$ are compatible. 

Fix $i\in n$.  Since $p_i\in \mathbf{P}^{L[g_i]}$, there is $x_i\in \RR^{L[g_i]}$ such that $p_i\in L[x_i]$. 
Because $g_i$ is $\mathbf{Q}$-generic over $L$, there is $\delta < \omega_1$ such that $\delta > \gamma$ and $x_i \in L[g_i\restriction\delta]$. 
Then $p_i\in L[g_i\restriction \delta]$. 
From Equation \eqref{eq: ultrafilter step5} and the homogeneity of $\mathbf{Q}$, we obtain 

\begin{equation}
	\text{$\mathbbm{1}$} _{\mathbf{Q}}\forc{\mathbf{C}(\omega_1\backslash \delta)}{L[g_i\restriction \delta]}
	\check{p}_i \forc{\mathbf{P}}{L[g_i\restriction\delta, \cdot]} \phi\left( \check{\check{\bar{r}}}_i, \check{
		\check{\vec{x}}}, \check{\check{\vec{\alpha}}},\dot{\mathcal{P}}\right).
\end{equation}

Now, consider $g_i\restriction\delta$ as one real.
Then we can think of $L[g]$ as a $\mathbf{Q}$-extension of $L[g_i\restriction \delta]$ (see Theorem \ref{thm:Cohen-properties}~\ref{th: prereq any real in Q extension is Q ground}).
Since $\mathbf{Q}\cong \mathbf{C}(\omega_1\backslash \delta)$, and $i$ was arbitrary, we have that for every $i\in n$,
\begin{equation}
	p_i \forc{\mathbf{P}}{L[g]} \phi \left(  \check{\bar{r}}_i, \check{\vec{x}}, \check{\vec{\alpha}},\dot{\mathcal{P}}\right). \\ \label{eq: ultrafilter step6}
\end{equation}
Remember that each $p_i$ extends $p\in L[g\restriction \gamma]$ for $i\in n$.
Now we have satisfied hypothesis~\ref{item: amalgamation th ultrafilter} in the definition of $(n,n-1)$-amalgamation, and in $L[g]$ the conditions $\{p_i\}_{i\in n}$ are compatible.
Nevertheless these conditions force incompatible statements, namely, $p_i$ forces that $\bar{r}_i$ is an element of the ultrafilter given by $\phi$.
This leads to a contradiction, since $\bigcap_{i\in n} \bar{r}_i$ is finite.

\end{proof}

\subsection{Funicular orders without the inaccessible}

Our goal is to show that for any fixed funicular preorder $\pbelow$ there is a model of set theory with a prelinearization of $\pbelow$ but without a nonprincipal ultrafilter over $\omega$. 
Theorem~\ref{th: no ultrafilter no inaccessible} gives us a family of models in which there are no nonprincipal ultrafilters over $\omega$, but we need to choose an appropriate forcing $\mathbf{P}$ that satisfies the hypothesis of the theorem and adds a corresponding prealinearization $\mathcal{P}$ consisting of approximations by prelinearizations in models of the form $L[x]$. 
Proving that $\mathcal{P}$ will have the properties we need is the role of the following lemma. 

\begin{lemma} \label{lemma: local forcing forces the object}
	Let $\mathbf{P}$ be a local forcing poset in $L[g]$ with $\psi(p)$ of the form
	\[\forall r\in \RR \forall s\in p^{<\omega} \exists t\in p^{<\omega} \Psi(r,s,t), \]
	where $\Psi$ is absolute between transitive models, and such that $\mathbf{P}$ is $\sigma$-closed and satisfies item~\eqref{item: def - extendability} of $(n,m)$-amalgamation.
	Let $h$ be a $\mathbf{P}$-generic filter over $L[g]$ and $\mathcal{P}=\cup h$. 
	Then 
	 \[L(\RR, \mathcal{P})^{L[g,h]}\models \mathcal{P}\subseteq \RR \text{ and } \psi(\mathcal{P}). \]
\end{lemma}
	 
\begin{proof}
	This is essentially the second part of the proof of \cite[Theorem~3.11]{fatalini-unit-circle-partition}, noticing that condition \ref{item: def - extendability} in the definition of $(n,m)$-amalgamation plays the role of \emph{extendability} in that paper. However, we include it here for completeness and because we have extended slightly the form of the formula $\psi$.
	
	Let $W$ be the model $L(\RR, \mathcal{P})^{L[g,h]}$. Then it is clear that $\RR^W=\RR^{L[g,h]}$. 
    Since $\mathbf{P}$ is $\sigma$-closed, $\RR^{L[g,h]}=\RR^{L{[g]}}$.

Notice that a formula $\psi$ as in the hypothesis is absolute between transitive models that agree on $\RR$ and contain $\mathcal{P}$ as an element. Since this is true for the models $L(\RR, \mathcal{P})^{L[g,h]}$ and $L[g,h]$, we just need to prove that $L[g,h]$ satisfies $\psi(\mathcal{P})$.
	
	Work inside $L[g,h]$.
	Since $h\subseteq \mathbf{P}$, we have that for all $p\in h$ there is some $x_p\in \RR$ such that 
	\[ L[x_p] \models p\subseteq \RR \text{ and }
\forall r\in \RR \forall s\in p^{<\omega} \exists t\in p^{<\omega} \Psi(r,s,t).  \]
	
From this and the fact that being a real number is absolute, we have that $\mathcal{P}\subseteq \RR$. 
Now, fix $r\in \RR$ and $s\in \mathcal{P}^{<\omega}$.
First, we claim that the set 
	\[D_r= \{ \bar{p}\in \mathbf{P} \mid r\in L[x_{\bar{p}}] \} \]
	is dense, which we prove as follows.
	Fix $p\in \mathbf{P}$.
	There is some $\beta < \omega_1$ such that $r \in L[g\restriction \beta]$. 
	By item~\eqref{item: def - extendability} of the definition of $(n,n+1)$-amalgamation, there is some $\gamma>\beta$ and $\bar{p}\leq_\mathbf{P} p$ such that $L[g\restriction \gamma]\models \text{``$\bar{p}\subseteq \RR$ and $\psi(\bar{p})$''}$.
	Notice that $r\in L[g\restriction\gamma]$ and $L[x_{\bar{p}}]=L[g\restriction\gamma]$ (by definability of $x_{\bar{p}}$ from $\bar{p}$ in the definition of local forcing).
	Therefore $D_r$ is dense.
	Secondly, let $s=\{s_0, \dots, s_{l-1}\}$ where $l\in \omega$.
	Then there is a finite set of conditions $p_0, \dots, p_{l-1}$ in $h$ such that $s_i\in p_i$ for all $i\in l$. 
	Since $h$ is a filter, there is $p\in h$ such that $p\leq_{\mathbf{P}} p_i$ for all $i\in l$. 
	In particular, $p_i \subseteq p$ for all $i\in l$ and $s\subseteq p$. 
	Because $h$ is a generic filter, $p\in h$, and $D_r$ is dense, there is $\bar{p}\leq_\mathbf{p} p$ such that $\bar{p}\in D_r$. 
	Since $s\in \bar{p}^{<\omega}$ and $r\in L[x_{\bar{p}}]$, we get that 
	\[L[x_{\bar{p}}]\models \exists t\in \bar{p}^{<\omega} \Psi(r,s,t).\]
	
		By absoluteness of $\Psi$ and noticing $t\in\bar{p}^{<\omega}\subseteq\mathcal{P}^{<\omega}$, we obtain that it is true (in $L[g,h]$) that $\exists t\in \mathcal{P}^{<\omega}\Psi(r,s,t)$. 
		Since $r$ and $s$ were arbitrary, we get that 
	\[ L(\RR, \mathcal{P})^{V[g,h]} \models \mathcal{P}\subseteq\RR\text{ and } \psi (\mathcal{P}),\]
	as we wanted to show.
\end{proof}

\begin{rmk}
	Notice that the formulas of the form $\exists r \in \RR \exists s\in p^{<\omega} \Psi(r,s)$ also work in the proof.
    However, the applications will be of the form
	\[\forall r\in \RR (\forall s\; \psi_1(r,s) \text{ and } \exists t\; \psi_2(r,t))\]
	which is equivalent to a formula as in the hypothesis of Lemma~\ref{lemma: local forcing forces the object}.
\end{rmk}

\begin{dfn}
We will say that $X$ is a  \emph{standard Polish space} if it is a countable product of $\omega^\omega, 2^\omega, [0,1],$ or $\RR$. 
Notice that they can all be coded by the reals in a natural way. 
\end{dfn}

\begin{dfn}
		Let $\mathbf{Q}$ be the finite support product of $\omega_1$-many copies of Cohen forcing, let $g$ be a $\mathbf{Q}$-generic filter over $L$. 
		Let $\pbelow$ be a Borel preorder on a standard Polish space $X$.
	Let $\Pprelin$ be the local forcing poset in $L[g]$ given by
	\[p \in \Pprelin \iff \exists x \in \RR\; L[x]\models \text{``$p\subseteq X\times X$ and  $p$ is a prelinearization of $\pbelow$,''} \]
	ordered by reverse inclusion. 
\end{dfn}

Notice that it is a local forcing, meaning that for every $p$ a real witness $x$ for being a condition is definable from $p$. 
This is due to the fact that $p$ is a total relation and we can assume $X=\RR$ since it is a standard Polish space. 

Moreover, if we denote by $\psi(p)$ the natural formula describing that $p$ is a prelinearization of $\pbelow$, then $\psi(p)$ is (equivalent to a formula) of the form of the hypothesis of Lemma~\ref{lemma: local forcing forces the object}.
More concretely, since
\begin{equation}
\begin{gathered}
    p \text{ is a prelinearization of }\pbelow \\
    \iff \\
    {\pbelow} \subseteq p, p\; \text{is a prelinear order and } \forall x,y\in X (x\pbelow y \land \neg (y\pbelow x)) \rightarrow (y,x)\not\in p,
    \end{gathered}
\end{equation}

we can write $\psi(p)$ as
\[\forall r_0, r_1 \in X \left(\exists t\in p\;\psi_0(r_0,r_1,t) \land \forall s \in p \;\psi_1(r_0,r_1,s) \land \psi_2 \right)
\]
where 
\begin{equation*}
\begin{gathered}
\psi_0(r_0,r_1,t) \text{ denotes $``r_0\pbelow r_1 \rightarrow (r_0,r_1)=t$,'' and}\\
\psi_1(r_0,r_1,s): \text{ denotes ``$(r_0 \pbelow r_1 \land \neg (r_1 \pbelow r_0)) \rightarrow (r_0,r_1)\neq s$,'' and}\\
\psi_2(p) \text{ denotes ``$p \text{ is a total preorder.}$''}
\end{gathered}
\end{equation*}

\begin{dfn}\label{def: cohen funicular}
	An Borel prelinear order $\pbelow$ on a Polish space $X$ is \emph{Cohen-funicular} over $V$ iff for every $g$, $h$ mutually $\mathbf{C}$-generic over $V$ and for every $x \in V[g]$ and $y \in V[h]$ with
	$x \pbelow y$, there is $z \in V$ with $x \pbelow z \pbelow y$.
\end{dfn}

Notice that a Borel prelinear order that is funicular as in Definition \ref{def:funicular} is Cohen-funicular. Therefore, all the preorders in Lemma~\ref{lemma: exs of funicular preorders} are also Cohen-funicular.  

\begin{cor}[Corollary of Theorem \ref{th: no ultrafilter no inaccessible}] \label{thm:local-prelinearization}
	Let $\mathbf{Q}$ be the finite support product of $\omega_1$-many copies of Cohen forcing, let $g$ be a $\mathbf{Q}$-generic filter over $L$. 
	Let $\pbelow$ be a Borel Cohen-funicular preorder and consider the local forcing $\mathbf{P}_{\pbelow}$.  
	Let $h$ be  $\mathbf{P}_{\pbelow}$-generic over $L[g]$ and $\mathcal{P}=\cup h$. Then
	\[L(\RR, \mathcal{P})^{L[g,h]}\models \mathsf{ZF}+\mathsf{DC}+\neg\Uf(\omega)+\mathcal{P} \text{ is a prelinearization of } \pbelow.\]
\end{cor}

\begin{proof}
	Naturally, we want to apply Theorem~\ref{th: no ultrafilter no inaccessible} and Lemma~\ref{lemma: local forcing forces the object} for the forcing $\Pprelin$. 
	We have established that $\Pprelin$ is a local forcing poset. We claim it is $\sigma$-closed. 
    Suppose $\{\pless_n\}_{n\in\omega}$ is a decreasing sequence of conditions in $\Pprelin$ and suppose that $x_n$ is a real witnessing $\pless_n$ is a condition for each $n\in \omega$.
    Thus in $L[x_n]$, $\pless_n$ is a prelinearization of $\pbelow$.
    Take $\pless^*=\bigcup_{n\in \omega} p_n$ and $x= \bigoplus_{n\in \omega} x_n$. 
    It is clear that 
    \[ L[x]\models \text{``${\pbelow} \subseteq {\pless^*} \subseteq X\times X$ and $\pless^*$ is a preorder.''} \]
    Moreover, fix $x,y\in \pless^*$ such that $x\pbelow y$ and $y\not \pbelow x$.
    There is $n\in \omega$ such that $(x,y)\in \pless_m$ for all $m\geq n$.  
    Since $\pless_m$ is a prelinearization of $\pbelow$ in the corresponding model, in particular $(y,x)\not\in \pless_m$ for all $m\geq n$. 
    Therefore $(y,x)\not \in \pless^*$. 
    In summary, $\pless^*$ is a preorder that extends $\pbelow$ and such that ${\psbelow} \subseteq {<^*}$. 
    By Lemma~\ref{lemma: prereq preorder can be extended to prelin}, in $L[x]$, $\pless^*$ can be extended to a prelinear order $\pless$ that is a prelinearization of $\pbelow$.
    Hence $\pless \in \Pprelin$ and $\Pprelin$ is $\sigma$-closed.
    
    Moreover, $\Pprelin$ satisfies $(3,2)$-amalgamation. For the first item, notice that for any condition $p=\pless$ whose locality is witnessed by the real $x$ and for any $\beta<\omega_1$, there is $\gamma\geq \beta$ such that $x\in L[g\restriction \gamma]$. 
    In $L[g\restriction \gamma]$, $\pless$ is  a preorder that extends $\pbelow$ and such that ${\psbelow} \subseteq {<}$. 
    Therefore, by Lemma~\ref{lemma: prereq preorder can be extended to prelin}, $\pless$ can be extended in $L[g\restriction\gamma]$ to a prelinearization of $\pbelow$.
	For the second item, the proof is analogous to the argument in the proof of item (1) in Lemma~\ref{lemma: funicular implies 32balanced}. 
	
Applying Theorem~\ref{th: no ultrafilter no inaccessible}, we get that 
\begin{equation} \label{eq: model has no ultrafilter}
	L(\RR,\mathcal{P})^{L[g,h]}\models\mathsf{ZF}+ \mathsf{DC}+\neg\Uf(\omega).
\end{equation}

Furthermore, recall that we can define the forcing $\Pprelin$ in a way that conforms to the form of $\psi(p)$ in the hypothesis of Lemma~\ref{lemma: local forcing forces the object} as described earlier. 
We can apply then Lemma~\ref{lemma: local forcing forces the object} and, together with~\eqref{eq: model has no ultrafilter}, we obtain: 
\[	L(\RR,\mathcal{P})^{L[g,h]}\models \mathsf{ZF}+\mathsf{DC}+\neg\Uf(\omega) + \psi(\mathcal{P}).\]	
\end{proof}

\section{Outlook}

This work just begins the study of the strength of the existence of prelinearizations of Borel preorders as choice principles, and naturally many open questions remain.

\begin{quest}
Let $X$ be a Polish space.
For which analytic preorders $\pbelow$ does the existence of a prelinearization imply the existence of a nonprincipal ultrafilter on $\omega$?
Or a Hamel basis for $\reals$ over $\rationals$?
When does the existence of a prelinearization in a transitive extension $M$ of $W$ imply the existence of a set of ordinals in $M \setminus W$?
\end{quest}

\begin{quest}
Is there a combinatorial property about preorders that characterizes funicularity? 
\end{quest}

\begin{quest}
For which Borel ideals $I$ over $\omega$ is the preorder $\le_I$ on $\reals^\omega$ defined for $x, y \in \reals^\omega$ by
\[ x \le_I y \Leftrightarrow \{ n < \omega : x(n) > y(n) \} \in I, \]
funicular?
This is of interest in social choice theory, because such orders arise naturally in the analysis of strengthenings of the Pareto efficiency condition.
\end{quest}

\emph{Acknowledgement.} The second author gratefully acknowledges the help provided by his advisors, Justin Moore and Paul Larson, in conducting this research and preparing this paper.\\

No data are associated with this article. For the purpose of open access, the authors have applied a CC-BY licence to any Author Accepted Manuscript version arising from this submission.

\small{
\bibliographystyle{siam}
\bibliography{math}
}
\end{document}